\documentclass[a4paper,english,oneside]{smfcustom}
\usepackage{mathrsfs,mathtools}
\usepackage[english]{babel}
\usepackage[utf8]{inputenc}
\usepackage[T1]{fontenc}
\usepackage{pxfonts,microtype,enumitem}
\usepackage{xcolor}\usepackage[all]{xy}
\usepackage{tikz}
\usetikzlibrary{patterns,arrows}
\def\C{\mathbb{C}}  \def\O{\mathcal{O}} 
 
\def\P{\mathbf{P}}
\def\Q{\mathbf{Q}}
\def\F{\mathbf{F}}
\def\G{\mathbf{G}}

\def\bydef{\coloneqq} 
\DeclareMathOperator{\rk}{rank}
\DeclareMathOperator{\codim}{codim}
\DeclareMathOperator{\Hom}{Hom}
\DeclarePairedDelimiter{\abs}{\lvert}{\rvert} 
\DeclarePairedDelimiter{\Set}{\lbrace}{\rbrace} 
\def\tprod{\mathop{\textstyle\prod}}

\let\leq\leqslant \let\geq\geqslant
\newcommand{\margin}[1]{\vspace\baselineskip #1\vspace\baselineskip\par\noindent}
\author[L.~Darondeau]{Lionel Darondeau}
\address{IMJ-PRG, Sorbonne Université, CNRS, Paris, France.}
\email{darondeau@imj-prg.fr}
\dedicatory{Dedicated to Professor Piotr Pragacz} 
\title {Isotropic Kempf--Laksov flag bundles}
\date{}
\begin{document}
\begin{abstract}
  We introduce analogs of the Kempf--Laksov desingularizations of Schubert bundles
  in (non-necessary Lagrangian) symplectic Grassmann bundles.
  In this setting, these are (possibly singular) irreducible flag bundles that are birational to Schubert bundles, and can be described as chains of zero-loci of regular sections in projectivized bundles.
  The orthogonal analogs are also presented.
  We immediatly derive universal Gysin formulas for isotropic Schubert bundles from these very constructions.
\end{abstract}
\maketitle

\section*{Introduction}
The goal of this short paper is to derive Gysin formulas for isotropic Schubert bundles (and for the isotropic Kempf--Laksov flag bundles introduced in this work) from the intersection theory developed in Fulton's book~\cite{Fulton}, using a plain line of thought based on the very definitions.
The elementary guiding idea is to obtain a (singular) birational model of Schubert bundles, by constructing isotropic vector spaces line by line.
To desingularize Schubert bundles, it is usual to use Bott--Samelson resolutions (see~\cite{Demazure}), but here we will deal with a much simpler geometry, in the spirit of Grothendieck's construction of flag bundles.
Sharing our goal of simplification, in~\cite{Kaz2}, Kazarian has also constructed an interesting birational model of Schubert bundles using zero-loci and projective bundles of lines, in the Lagrangian case, working with the Grassmannian as a base.
It seems however that our construction is more easily adapted to non-Lagrangian case, and it is our goal to work with base \(X\) (see~\cite{DP2} for a justification for this point).
In~\cite{AF18}, Anderson and Fulton also draw inspiration from Kempf--Laksov and Kazarian in order to prove formulas for a large class of degeneracy loci in classical types.
Some of the results of this work were announced in~\cite{DPNote}.

The paper is organized as follows. We first deal with the symplectic setting.
In Sect.~\ref{se:Schub}, we quickly recall a definition of Schubert bundles in isotropic Grassmann bundles, and we fix notations. 
In Sect.~\ref{se:KL}, we define flag bundles birational to isotropic Schubert bundles.
These are analogous to the flag bundles of Kempf and Laksov in~\cite{KL} which desingularize the Schubert bundles when working with the general linear groups.
These are constructed as a chain of zero-loci in projective bundles of lines. It is noteworthy that these are not smooth in general, but always cut out by regular sections, which is sufficient regarding our goals.
In Sect.~\ref{se:Gysin}, we derive Gysin formulas for isotropic Kempf--Laksov bundles and isotropic Schubert bundles.
In Sect.~\ref{se:Orthogonal}, we indicate how to adapt the arguments in order to treat the orthogonal setting.

\section{Schubert bundles}
\label{se:Schub}
Let \((E,\omega)\to X\) be a rank \(2n\) symplectic vector bundle for the symplectic form \(\omega\colon E\otimes E\to L\) with value in a line bundle \(L\to X\), over a variety \(X\).
For \(d\in\{1,\dotsc,n\}\), let \(\G_{d}^{\omega}(E)\) be the Grassmann bundle of isotropic \(d\)-planes in the fibers of \(E\).
For a vector space \(V\in E(x)\) let denote \(V^{\omega}\) its symplectic complement.
Let
\[
  0=E_{0}\subsetneq E_{1}\subsetneq\dotsb\subsetneq E_{n}=(E_{n})^{\omega}\subsetneq\dotsb\subsetneq(E_{0})^{\omega}=E
\]
be a reference flag of isotropic subbundles and co-isotropic subbundles of \(E\), where \(\rk(E_{i})=i\). For the sake of uniformity of notation, for \(i=0,1,\dotsc,n\), denote as well \(E_{2n-i}\bydef(E_{i})^{\omega}\).

A \textsl{partition} \(\lambda\) of a non-negative integer \(n\in\mathbb{N}\) is a decomposition of \(n\) as a sum of non-negative integers. 
We denote by \(\abs{\lambda}\) the number partitioned by \(\lambda\).
Two partitions are identified if they are the same up to order of the summands, or if one can be obtained from the other by adding some zeros.
It is usual to sort the summands in decreasing order and to write partitions as weakly decreasing sequences of non-negative integers.
A non-zero summand in a partition is called a \textsl{part}.
A partition is said \textsl{strict} if all its parts are distinct.
The \textsl{Young diagram} associated to a partition \(\lambda_1\geq \lambda_2\geq\dotsb\geq \lambda_d>0\) is the finite collection of cells, arranged in left-justified rows, with row lengths \(\lambda_1,\dotsc,\lambda_d\) (from top to bottom).
Inclusion of Young diagrams defines a partial order denoted \(\subseteq\) on the set of partitions  (it coincides with the restriction of the product order on \(\mathbb{N}^{(\mathbb{N})}\) to the set of weakly decreasing sequences).

Let \(\rho\bydef(2n,\dotsc,2n-d+1)\) be the maximal strict partition in \((2n)^{d}\bydef(2n,\dotsc,2n)\).
For a partition \(\lambda\subseteq(2n-d)^{d}\), we denote by \(\mu=\lambda^{\complement}\subseteq\rho\) the complementary partition of \(\lambda\) in \(\rho\), i.e. the partition with parts
\(
  \mu_{i}
  \bydef
  \rho_{i}
  -
  \lambda_{d+1-i}
\).
It is a strict partition with \(d\) parts.
One shall consider only \textsl{admissible} partitions \(\mu\), \textit{i.e.} partitions such that \(\mu_{i}+\mu_{j}\neq 2n+1\) for \(1\leq i,j\leq d\).

When the strict partition \(\mu=\lambda^{\complement}\) is admissible, there is the \textsl{Schubert open cell} \(\mathring{\Omega}_{\lambda}(E_{\bullet})\) in \(\G_{d}^{\omega}(E)\) given over the point \(x\in X\) by the conditions
\[
  \mathring{\Omega}_{\lambda}(E_{\bullet})
  \bydef
  \Set*{
    V\in\G_{d}^{\omega}(E)(x)
    \colon
    \dim\big(V\cap E_{\mu_{i}}(x)\big)=d+1-i,
    \text{ for }i=1,\dotsc,d
  }.
\]
The \textsl{Schubert bundle} \(\varpi_{\lambda}\colon \Omega_{\lambda}\to X\) is the Zariski-closure of \(\mathring{\Omega}_{\lambda}\), given over a point \(x\in X\) by the conditions
\[
  \Omega_{\lambda}(E_{\bullet})
  \bydef
  \Set*{
    V\in\G_{d}^{\omega}(E)(x)
    \colon
    \dim\big(V\cap E_{\mu_{i}}(x)\big)\geq d+1-i,
    \text{ for }i=1,\dotsc,d
  }.
\]
Observe that the incidence conditions are trivial if \(\lambda=0\), i.e. if \(\mu=\rho\).
The Schubert bundle \(\Omega_{\lambda}\)  is a subvariety of the Grassmann bundle \(\G_{d}^{\omega}(E)\), that is in general singular (\cite{Demazure}).
In the spirit of Kempf and Laksov~\cite{KL}, and also inspired by Kazarian~\cite{Kaz2}, we will now construct flag bundles \(\vartheta_{\mu}\colon F_{\mu}(E_{\bullet})\to X\) birational to Schubert bundles \(\varpi_{\lambda}\colon\Omega_{\lambda}\to X\).

\section{Isotropic Kempf--Laksov flag bundles}
\label{se:KL}
Let \(\F^{\omega}(1,\dotsc,d)(E)\) denote the bundle of flags of nested isotropic subspaces
with respective dimensions \(1,\dotsc,d\) in the fibers of \(E\).
Consider an admissible strict partition \(\mu=(\mu_{1},\dotsc,\mu_{d})\subseteq\rho\) with \(d\) parts.
We define the \textsl{isotropic Kempf--Laksov flag bundle} \(\vartheta_{\mu}\colon F_{\mu}(E_{\bullet})\to X\) as the subvariety given over the point \(x\in X\) by
\[
  F_{\mu}(E_{\bullet})(x)
  \bydef
  \Set*{
    \Set{0}=V_{0}\subsetneq V_{1}\subsetneq\dotsb\subsetneq V_{d}\in \F^{\omega}(1,\dotsc,d)(E)(x)
    \colon
    V_{d+1-i}\subseteq E_{\mu_{i}}(x)
  }.
\]

For \(\mu=\lambda^{\complement}\), the natural forgetful map \(\F^{\omega}(1,\dotsc,d)(E)\to\G_{d}^{\omega}(E)\) restricts to a map from the Kempf--Laksov bundle \(F_{\mu}\) to \(\Omega_{\lambda}\), which is invertible on the Schubert cell \(\mathring\Omega_{\lambda}\).
However, notice that in opposition to type \(A\), isotropic Kempf--Laksov flag bundles can be singular, as we shall soon illustrate (see Example~\ref{exem:sing}).

Let us first introduce some important combinatorial quantities attached to the partition \(\mu\).
For an admissible strict partition \(\mu\subseteq\rho\), introduce the \(d\) integers:
\begin{equation}
  \label{eq:delta}
  \delta_{i}
  \bydef
  \#\Set{
    j>i
    \colon
    \mu_{i}+\mu_{j}<2n+1
  }.
\end{equation}
For a flag \(V_{\bullet}\in F_{\mu}(E_{\bullet})\), since \(V_{d+1-j}\subseteq E_{\mu_{j}}\), as soon as \(E_{\mu_{j}}\subseteq(E_{\mu_{i}})^{\omega}\), \textit{i.e.} as long as \(\mu_{i}+\mu_{j}<2n+1\), one has 
\(V_{d+1-j}\subseteq(E_{\mu_{i}})^{\omega}\).
Therefore, for \(i=1,\dotsc,d\) one has \(V_{\delta_{i}}\subseteq V_{d-i}\cap(E_{\mu_{i}})^{\omega}\). 
In our step-by-step construction of isotropic Kempf--Laksov flag bundles,
we will soon see that
\begin{equation}
  \label{eq:incidence}
  V_{d-i}\cap(E_{\mu_{i}})^{\omega}
  =
  V_{\delta_{i}}
\end{equation}
is an expected incidence condition.

Let us for now study two elementary cases of singular Kempf--Laksov flag bundles, and the role played by \eqref{eq:incidence} in these examples. There is no such example for \(n<3\), and these are the two only examples for \(n=3\).
\begin{exem}[\(n=3\)]
  \label{exem:sing}
  For \(X=\{\text{pt}\}\), consider the vector space \(E=\C^{6}\), equipped with a symplectic basis \((e_{1},e_{2},e_{3})\) of \(\C^{3}\) and dual vectors \((f_{1},f_{2},f_{3})\). Let \(E_{\bullet}\) be the standard symplectic flag corresponding to the basis \((e_{1},e_{2},e_{3},f_{3},f_{2},f_{1})\).
  Note that the forgetful map \(\F^{\omega}(1,\dots,d)\to\F^{\omega}(1,\dotsc,d-1)\) induces (dominant) maps \(F_{(\mu_{1},\dotsc,\mu_{d})}(E_{\bullet})\to F_{(\mu_{2},\dotsc,\mu_{d})}(E_{\bullet})\).
  \begin{itemize}
    \item
      Take \(\mu=(6,5,3)\).
      \begin{itemize}
        \item
          The fiber of \(F_{5,3}(E_{\bullet})\) over \([e_{1}]\in F_{3}(E_{\bullet})\) is \(\P(\langle e_{2},e_{3},f_{3},f_{2}\rangle)\), which is a \(3\)-dimensional projective space, whereas the fiber of \(F_{5,3}(E_{\bullet})\) over any other line \([a e_{1}+b e_{2}+c e_{3}]\), with \(\abs{b}^{2}+\abs{c}^{2}=1\), identifies with
          \(\P(\langle e_{1},e_{2},e_{3},f_{3},f_{2}\rangle/\langle ae_{1}+be_{2}+ce_{3},\bar{c}f_{3}+\bar{b}f_{2}\rangle)\), which is a \(2\)-dimensional projective space.
        \item
          In both cases the fiber of \(F_{6,5,3}(E_{\bullet})\) over \(F_{5,3}(E_{\bullet})\) is then a \(1\)-dimensional projective space.
      \end{itemize}

      Since \((E_{5})^{\omega}=E_{1}\) and \((E_{6})^{\omega}=E_{0}\), in this case the expected incidence conditions are \(V_{1}\cap E_{1}=\Set{0}\) and \(V_{2}\cap E_{0}=\Set{0}\) (the second condition is empty).
    \item
      Take \(\mu=(6,5,4)\).
      \begin{itemize}
        \item
          The fiber of \(F_{5,4}(E_{\bullet})\) over \([e_{1}]\in F_{4}(E_{\bullet})\) is \(\P(\langle e_{2},e_{3},f_{3},f_{2}\rangle)\), which is a \(3\)-dimensional projective space, whereas the fiber of \(F_{5,4}(E_{\bullet})\) over any other line \([a e_{1}+b e_{2}+c e_{3}+d f_{3}]\), with \(\abs{b}^{2}+\abs{c}^{2}+\abs{d}^{2}=1\), identifies with
          \(\P(\langle e_{1},e_{2},e_{3},f_{3},f_{2}\rangle/\langle ae_{1}+be_{2}+ce_{3},\bar{d}e_{3}+\bar{c}f_{3}+\bar{b}f_{2}\rangle)\), which is a \(2\)-dimensional projective space.
        \item
          In both cases the fiber of \(F_{6,5,4}(E_{\bullet})\) over \(F_{5,4}(E_{\bullet})\) is then a \(1\)-dimensional projective space.
      \end{itemize}

      In this case the expected incidence conditions are still \(V_{1}\cap E_{1}=\Set{0}\) and \(V_{2}\cap E_{0}=\Set{0}\).
  \end{itemize}
\end{exem}

From these examples, one can extrapolate the role played by~\eqref{eq:incidence} in the general case.
However, in order to reduce the combinatorial difficulty, we will deal with a strengthening of the conditions~\eqref{eq:incidence}.
For a flag \(V_{\bullet}\in F_{\mu}(E_{\bullet})(x)\) over a point \(x\in X\), denote by \(\nu(V_{\bullet})\subseteq\rho\) the strict partition made of the \(d\) integers
\begin{equation}
  \label{eq:nu}
  \Set*{
    \nu_{i}\geq 1
    \colon
    \dim\left(V_{d}\cap E_{\nu_{i}}(x)\right)
    >
    \dim\left(V_{d}\cap E_{\nu_{i}-1}(x)\right)
  }.
\end{equation}
The strictness follows from Gaussian elimination in a frame of \(V_{d}\) with respect to the reference flag \(E_{\bullet}(x)\).
This argument also implies that \(\nu(V_{\bullet})\subseteq\mu\). Lastly, one also infers from this argument that, since \(V_{d}\) is isotropic, the partition \(\nu(V_{\bullet})\) has to be admissible. Indeed, if not, after the above Gaussian elimination, two vectors of the frame of \(V_{d}\) would not be symplectically orthogonal.
We denote by
\begin{equation}
  \label{eq:F°}
  \mathring{F}_{\mu}(E_{\bullet})
  \bydef
  \Set{V_{\bullet}\in F_{\mu}(E_{\bullet})\colon \nu(V_{\bullet})=\mu},
\end{equation}
the set of flags with \(\nu(V_{\bullet})\) maximal.
Clearly, for a flag in \(\mathring{F}_{\mu}(E_{\bullet})\), the conditions~\eqref{eq:incidence} hold.

\begin{theo}
  \label{theo:KL}
  Let \(\mu\subseteq\rho\) be an admissible strict partition.
  Fix a point \(x\in X\).
  \begin{itemize}
    \item
      The isotropic Kempf--Laksov flag bundle \(F_{\mu}(E_{\bullet})(x)\) is a variety of dimension \(\abs{\mu}+\abs{\delta}-d^{2}\), that can be described as a chain of zero-loci of regular sections in projectivized bundles.
    \item
      The subvariety \(\mathring{F}_{\mu}(E_{\bullet})(x)\), is an irreducible open dense subset contained in the non-singular part of \(F_{\mu}(E_{\bullet})(x)\).
  \end{itemize}
\end{theo}
\begin{proof}
  The idea is to construct the isotropic flag \(V_{1}\subsetneq\dotsb\subsetneq V_{d}\) line-by-line (considering quotients \(V_{i}/V_{i-1}\) of successive spaces), in such way that it satisfies the incidence conditions defining \(F_{\mu}(E_{\bullet})\) at each step.

  We proceed by double induction on \(d\) and \(\abs{\mu}\). For \(d=1\), the isotropic Kempf--Laksov flag bundle \(F_{\mu_{1}}(E_{\bullet})\) is \(\P(E_{\mu_{1}})\). For \(\abs{\mu}=d+\dotsb+1\), minimal, \(\mu=(d,\dotsc,1)\). Therefore, \(F_{\mu}(E_{\bullet})\) is a point, and \(\delta=(d-1,\dotsc,1,0)\). In both cases, the result is straightforward. 

  We now describe the step \(F_{(\mu_{1},\mu_{2},\dotsc,\mu_{d})}(E_{\bullet})\to F_{(\mu_{2},\dotsc,\mu_{d})}(E_{\bullet})\).
  Let \(U_{d-1}\) be the universal subbundle of rank \(d-1\) on \(\F^{\omega}(1,\dotsc,d-1)(E)\).
  Note that in restriction to \(F_{(\mu_{2},\dotsc,\mu_{d})}(E_{\bullet})\):
  \begin{itemize}
    \item the condition \(V_{d-1}\subseteq E_{\mu_{2}}(x)\) yields: \(U_{d-1}\subseteq E_{\mu_{2}}\subseteq E_{\mu_{1}}\);
    \item the condition \(V_{d-1}\) isotropic yields:
      \(V_{d-1}\oplus \ell(x)\) isotropic \(\Leftrightarrow \ell\subset(U_{d-1})^{\omega}\) (recall that a line \(\ell\) is always isotropic).
  \end{itemize}
  It thus follows from the definition of \(F_{\mu}(E_{\bullet})\) that
  \[
    F_{(\mu_{1},\dotsc,\mu_{d})}(E_{\bullet})
    \simeq
    \Set*{
      \ell\in\P(E_{\mu_{1}}/U_{d-1})
      \colon
      \ell\subseteq(U_{d-1})^{\omega}
    }.
  \]
  In the above quotient, and in the rest of the text, we only imply the restriction of \(E_{\mu_{1}}\) and of \(U_{d-1}\) to \(F_{(\mu_{2},\dotsc,\mu_{d})}(E_{\bullet})\), allowing the expression to make sense.
  We denote by \(U_{d}/U_{d-1}\) the tautological subbundle of \(\P(E_{\mu_{1}}/U_{d-1})\), so that \(U_{d}\) coincide with the restriction to \(F_{(\mu_{1},\dotsc,\mu_{d})}(E_{\bullet})\) of the universal subbundle of rank \(d\) on \(\F^{\omega}(1,\dotsc,d)(E)\).
  Since we restrict to \(F_{(\mu_{2},\dotsc,\mu_{d})}(E_{\bullet})\), one has
  \(U_{d-1}\subseteq(U_{d-1})^{\omega}\)
  and
  \(U_{\delta_{1}}\subseteq U_{d-1}\cap(E_{\mu_{1}})^{\omega}\).
  Hence, there is a well-defined global section \(s\) of the vector bundle
  \[
    \Hom\big(U_{d}/U_{d-1},L\otimes (U_{d-1}/U_{\delta_{1}})^{\vee}\big)
    \simeq
    L
    \otimes
    (U_{d}/U_{d-1})^{\vee}
    \otimes
    (U_{d-1}/U_{\delta_{1}})^{\vee}
  \]
  defined at the point \(\ell=V_{d}/V_{d-1}\subseteq E_{\mu_{1}}/V_{d-1}\) by:
  \[
    s(\ell)
    \bydef
    \big\{t\in\ell\mapsto\omega(t,\cdot)\rvert_{V_{d-1}}\big\}.
  \]
  We denote by \(Z_{d}\) the zero-locus of \(s\) in \(\P(E_{\mu_{1}}/U_{d-1})\).
  Over a point \(V_{d-1}\supsetneq\dotsb\supsetneq V_{1}\), the lines in \(Z_{d}\) are these lines that are (symplectically) orthogonal to \(V_{d-1}\) or equivalently the lines \(\ell\) such that the vector space \( V_{d}=\ell\oplus V_{d-1} \) is isotropic. Indeed, both \(\ell\) and \(V_{d-1}\) are already isotropic.
  Therefore,
  \(
  F_{(\mu_{1},\dotsc,\mu_{d})}(E_{\bullet})
  \simeq
  Z_{d}
  \).

  Over a point \(V_{d-1}\supsetneq\dotsb\supsetneq V_{1}\) above a point \(x\in X\), the zero-locus \(Z_{d}\) consists of the common zeroes of the linear forms
  \[
    \omega(\cdot,\ell')
    \colon
    {E_{\mu_{1}}(x)/V_{d-1}}
    \to
    L(x),
  \]
  for \(\ell'\subset V_{d-1}\). Such a linear form is trivial if and only if
  \(
  \ell'
  \subseteq
  V_{d-1}\cap (E_{\mu_{1}})^{\omega}(x)
  \).
  As a consequence the codimension
  \(\codim(Z_{d},\P(E_{\mu_{1}}/U_{d-1}))\)
  of the fibers of \(Z_{d}\) is given at a point \(V_{d-1}\supsetneq\dotsb\supsetneq V_{1}\in F_{(\mu_{2},\dotsc,\mu_{d})}(E_{\bullet})(x)\) by
  \[
    \codim(Z_{d},\P(E_{\mu_{1}}/U_{d-1}))
    =
    \dim(V_{d-1})
    -
    \dim(V_{d-1}\cap (E_{\mu_{1}})^{\omega}(x)).
  \]
  Now, recall that
  \(
  U_{\delta_{1}}
  \subseteq
  U_{d-1}\cap (E_{\mu_{1}})^{\omega}
  \).
  One infers the following upper bound on the codimension of \(Z_{d}\) at a point \(V_{\bullet}\):
  \[
    \tag{\(*\)}\label{eq:codim}
    \codim(Z_{d},\P(E_{\mu_{1}}/U_{d-1}))
    \leq
    \dim(V_{d-1})
    -
    \dim(V_{\delta_{1}})
    =
    \rk\big(
      L
      \otimes
      (U_{d}/U_{d-1})^{\vee}
      \otimes
      (U_{d-1}/U_{\delta_{1}})^{\vee}
    \big).
  \]
  Note that this inequality is an equality if and only if \(V_{d-1}\cap(E_{\mu_{1}})^{\omega}(x)=V_{\delta_{1}}\). This is condition~\eqref{eq:incidence} for \(i=1\).

  The flags in \(\mathring{F}_{\mu}(E_{\bullet})\) satisfy all conditions~\eqref{eq:incidence}, for \(i\leq d\).
  One infers inductively that the dimension of \(\mathring{F}_{\mu}(E_{\bullet})\) is
  \[
    \big(\mu_{1}-d\big) -\big((d-1)-\delta_{1}\big)
    +
    \big((\mu_{2}+\dotsb+\mu_{d})+(\delta_{2}+\dotsb+\delta_{d})-(d-1)^{2}\big)
    =
    \abs{\mu}+\abs{\delta}-d^{2}.
  \]
  Moreover, it follows from~\eqref{eq:codim} that the dimensions of the irreducible components of \(F_{\mu}(E_{\bullet})\simeq Z_{d}\) are all at least equal to the dimension of \(\mathring{F}_{\mu}(E_{\bullet})\).

  Now, consider an admissible partition \(\nu\subsetneq\mu\), for which there are \(p\) elements in the chain of admissible partitions from \(\nu\) to \(\mu\), with respect to the partial order \(\subseteq\).
  On the one hand, by Lemma~\ref{lemm:fibration} below, one has:
  \[
    \dim\left(\Set*{V_{\bullet}\in F_{\mu}(E_{\bullet})\colon\nu(V_{\bullet})=\nu}\right)
    <
    p+\dim\left(\mathring{F}_{\nu}(E_{\bullet})\right)
    =
    p+\abs{\nu}+\abs{\delta(\nu)}-d^{2}.
  \]
  On the other hand, by Lemma~\ref{lemm:dimension}, one has:
  \[
    \dim\left(\Set*{V_{\bullet}\in F_{\mu}(E_{\bullet})\colon\nu(V_{\bullet})=\mu}\right)
    =
    \dim\left(\mathring{F}_{\mu}(E_{\bullet})\right)
    =
    \abs{\mu}+\abs{\delta(\mu)}-d^{2}
    =
    p+\abs{\nu}+\abs{\delta(\nu)}-d^{2}.
  \]
  Putting these facts together, one infers that
  \[
    \dim\left(\Set*{V_{\bullet}\in F_{\mu}(E_{\bullet})\colon\nu(V_{\bullet})=\nu}\right)
    <
    \dim\left(\Set*{V_{\bullet}\in F_{\mu}(E_{\bullet})\colon\nu(V_{\bullet})=\mu}\right).
  \]
  Thus
  \[
    \dim\left(\Set*{V_{\bullet}\in F_{\mu}(E_{\bullet})\colon\nu(V_{\bullet})\subsetneq\mu}\right)
    <
    \dim\left(\Set*{V_{\bullet}\in F_{\mu}(E_{\bullet})\colon\nu(V_{\bullet})=\mu}\right).
  \]
  Therefore \(\mathring{F}_{\mu}(E_{\bullet})=\Set*{V_{\bullet}\in F_{\mu}(E_{\bullet})\colon\nu(V_{\bullet})=\mu}\) is dense in \(F_{\mu}(E_{\bullet})\).

  Let us now prove by induction on \(d\) that \(\mathring{F}_{\mu}(E_{\bullet})\) is irreducible. The regularity of \(s\) will follow.
  The same argument also gives the smoothness of \(\mathring{F}_{\mu}(E_{\bullet})\).
  For \(d=1\), one has \(\mathring{F}_{\mu_{1}}(E_{\bullet})=F_{\mu_{1}}(E_\bullet)=\P(E_{\mu_{1}})\). Hence \(\mathring{F}_{\mu}(E_{\bullet})\) is irreducible (and smooth).
  The forgetful map \(\F^{\omega}(1,\dotsc,d)(E)\to\F^{\omega}(1,\dotsc,d-1)(E)\) gives a smooth fibration from \(\mathring{F}_{\mu}(E_{\bullet})\subseteq F_{\mu}(E_{\bullet})\) to \(\mathring{F}_{(\mu_{2},\dotsc,\mu_{d})}(E_{\bullet})\subseteq F_{(\mu_{2},\dotsc,\mu_{d})}(E_{\bullet})\).
  Assuming that \(\mathring{F}_{(\mu_{2},\dotsc,\mu_{d})}(E_{\bullet})\) is irreducible (and smooth), one gets that \(\mathring{F}_{\mu}(E_{\bullet})\) as well is irreducible (and smooth).
\end{proof}

It remains to prove the following combinatorial lemmas on admissible strict partitions in order to conclude.
\begin{lemm}
  \label{lemm:dimension}
  Let \(\mu\) and \(\nu\) be two admissible strict partitions in \(\rho\). If \(\nu\) is a direct predecessor of \(\mu\) for the product order among all admissible partitions, then
  \[
    \abs{\nu}+\abs{\delta(\nu)}
    =
    \abs{\mu}+\abs{\delta(\mu)}-1.
  \]
\end{lemm}
\begin{proof}
  In order to be more synthetic, for an admissible strict partition \(\mu\subseteq\rho\) we represent \(\Set{1,\dotsc,2n}\) by \(2n\) balls. We colorize the ball \(i\) in black if \(i\) is a part of \(\mu\), in gray if \((2n+1-i)\) is a part of \(\mu\), and in white otherwise.
  Then there are three ways to obtain a direct predecessor \(\nu\) of \(\mu\):
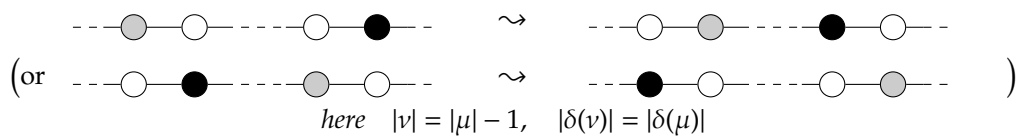
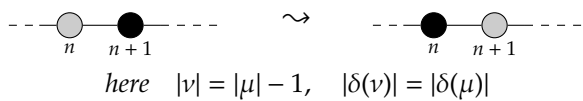
  \begin{enumerate}
    \usetikzlibrary{decorations.pathreplacing}
  \item
    \[
      \begin{tikzpicture}[scale=.4,baseline=(A)]
        \def\a{1}
        \tikzset{dynkin/.style={circle,draw,minimum size=2mm}}
        \path
          (0:0)   coordinate (A) ++(0:\a) coordinate (B)
          ++(0:\a)   node[dynkin,fill=black!20] (N1) {} 
          ++(0:2*\a) node[dynkin] (N2) {}
          ++(0:\a)    coordinate (C) ++(0:2*\a) coordinate (D)
          ++(0:\a) node[dynkin] (N3) {}
          ++(0:2*\a) node[dynkin,fill=black] (N4) {}
          ++(0:\a)    coordinate (E) ++(0:\a) coordinate (F);
        \draw[dashed] (A)--(B) (C)--(D) (E)--(F);
        \draw (B)--(N1)--(N2)--(C) (D)--(N3)--(N4)--(E);
      \end{tikzpicture}
      \qquad\leadsto\qquad
      \begin{tikzpicture}[scale=.4,baseline=(A)]
        \def\a{1}
        \tikzset{dynkin/.style={circle,draw,minimum size=2mm}}
        \path
          (0:0)   coordinate (A) ++(0:\a) coordinate (B)
          ++(0:\a)   node[dynkin] (N1) {} 
          ++(0:2*\a) node[dynkin,fill=black!20] (N2) {}
          ++(0:\a)    coordinate (C) ++(0:2*\a) coordinate (D)
          ++(0:\a) node[dynkin,fill=black] (N3) {}
          ++(0:2*\a) node[dynkin] (N4) {}
          ++(0:\a)    coordinate (E) ++(0:\a) coordinate (F);
        \draw[dashed] (A)--(B) (C)--(D) (E)--(F);
        \draw (B)--(N1)--(N2)--(C) (D)--(N3)--(N4)--(E);
      \end{tikzpicture}
    \]
    \[
      \left(
        \text{or\quad}
        \begin{tikzpicture}[scale=.4,baseline=(A)]
          \def\a{1}
          \tikzset{dynkin/.style={circle,draw,minimum size=2mm}}
          \path
            (0:0)   coordinate (A) ++(0:\a) coordinate (B)
            ++(0:\a)   node[dynkin] (N1) {} 
            ++(0:2*\a) node[dynkin,fill=black] (N2) {}
            ++(0:\a)    coordinate (C) ++(0:2*\a) coordinate (D)
            ++(0:\a) node[dynkin,fill=black!20] (N3) {}
            ++(0:2*\a) node[dynkin] (N4) {}
            ++(0:\a)    coordinate (E) ++(0:\a) coordinate (F);
          \draw[dashed] (A)--(B) (C)--(D) (E)--(F);
          \draw (B)--(N1)--(N2)--(C) (D)--(N3)--(N4)--(E);
        \end{tikzpicture}
        \qquad\leadsto\qquad
        \begin{tikzpicture}[scale=.4,baseline=(A)]
          \def\a{1}
          \tikzset{dynkin/.style={circle,draw,minimum size=2mm}}
          \path
            (0:0)   coordinate (A) ++(0:\a) coordinate (B)
            ++(0:\a)   node[dynkin,fill=black] (N1) {} 
            ++(0:2*\a) node[dynkin] (N2) {}
            ++(0:\a)    coordinate (C) ++(0:2*\a) coordinate (D)
            ++(0:\a) node[dynkin] (N3) {}
            ++(0:2*\a) node[dynkin,fill=black!20] (N4) {}
            ++(0:\a)    coordinate (E) ++(0:\a) coordinate (F);
          \draw[dashed] (A)--(B) (C)--(D) (E)--(F);
          \draw (B)--(N1)--(N2)--(C) (D)--(N3)--(N4)--(E);
        \end{tikzpicture}
        \text{\phantom{or\quad}}
    \right)\]
    \[
      \textit{here}
      \quad
      \abs{\nu}=\abs{\mu}-1,
      \quad
      \abs{\delta(\nu)}=\abs{\delta(\mu)}
    \]
  \item[2a.]
    \[
      \begin{tikzpicture}[scale=.4,baseline=(A)]
        \def\a{1}
        \tikzset{dynkin/.style={circle,draw,minimum size=2mm}}
        \path
          (0:0)   coordinate (A) ++(0:\a) coordinate (B)
          ++(0:\a)   node[dynkin,fill=black!20] (N1) {}  +(-90:.75) node[scale=.75]{\(n\)}
          ++(0:2*\a) node[dynkin,fill=black] (N2) {} +(-90:.75) node[scale=.75]{\(n+1\)}
          ++(0:\a)    coordinate (C) ++(0:2*\a) coordinate (D);
        \draw[dashed] (A)--(B) (C)--(D);
        \draw (B)--(N1)--(N2)--(C);
      \end{tikzpicture}
      \qquad\leadsto\qquad
      \begin{tikzpicture}[scale=.4,baseline=(A)]
        \def\a{1}
        \tikzset{dynkin/.style={circle,draw,minimum size=2mm}}
        \path
          (0:0)   coordinate (A) ++(0:\a) coordinate (B)
          ++(0:\a)   node[dynkin,fill=black] (N1) {} +(-90:.75) node[scale=.75]{\(n\)}
          ++(0:2*\a) node[dynkin,fill=black!20] (N2) {} +(-90:.75)  node[scale=.75]{\(n+1\)}
          ++(0:\a)    coordinate (C) ++(0:2*\a) coordinate (D);
        \draw[dashed] (A)--(B) (C)--(D);
        \draw (B)--(N1)--(N2)--(C);
      \end{tikzpicture}
    \]
    \[
      \textit{here}
      \quad
      \abs{\nu}=\abs{\mu}-1,
      \quad
      \abs{\delta(\nu)}=\abs{\delta(\mu)}
    \]
  \item[2b.] For \(i=1,\dotsc,n-1\)
    
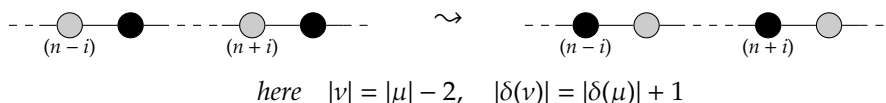
\[
      \begin{tikzpicture}[scale=.4,baseline=(A)]
        \def\a{1}
        \tikzset{dynkin/.style={circle,draw,minimum size=2mm}}
        \path
          (0:0)   coordinate (A) ++(0:\a) coordinate (B)
          ++(0:\a)   node[dynkin,fill=black!20] (N1) {}  +(-90:.75) node[scale=.75]{\((n-i)\)}
          ++(0:2*\a) node[dynkin,fill=black] (N2) {} +(-90:.75)
          ++(0:\a)    coordinate (C) ++(0:2*\a) coordinate (D)
          ++(0:\a) node[dynkin,fill=black!20] (N3) {} +(-90:.75) node[scale=.75]{\((n+i)\)}
          ++(0:2*\a) node[dynkin,fill=black] (N4) {} +(-90:.75)
          ++(0:\a)    coordinate (E) ++(0:\a) coordinate (F);
        \draw[dashed] (A)--(B) (C)--(D) (E)--(F);
        \draw (B)--(N1)--(N2)--(C) (D)--(N3)--(N4)--(E);
      \end{tikzpicture}
      \qquad\leadsto\qquad
      \begin{tikzpicture}[scale=.4,baseline=(A)]
        \def\a{1}
        \tikzset{dynkin/.style={circle,draw,minimum size=2mm}}
        \path
          (0:0)   coordinate (A) ++(0:\a) coordinate (B)
          ++(0:\a)   node[dynkin,fill=black] (N1) {}  +(-90:.75) node[scale=.75]{\((n-i)\)}
          ++(0:2*\a) node[dynkin,fill=black!20] (N2) {} +(90:.75)
          ++(0:\a)    coordinate (C) ++(0:2*\a) coordinate (D)
          ++(0:\a) node[dynkin,fill=black] (N3) {} +(-90:.75) node[scale=.75]{\((n+i)\)}
          ++(0:2*\a) node[dynkin,fill=black!20] (N4) {}
          ++(0:\a)    coordinate (E) ++(0:\a) coordinate (F);
        \draw[dashed] (A)--(B) (C)--(D) (E)--(F);
        \draw (B)--(N1)--(N2)--(C) (D)--(N3)--(N4)--(E);
      \end{tikzpicture}
    \]
    \[
      \phantom{1+{}}
      \textit{here}
      \quad
      \abs{\nu}=\abs{\mu}-2,
      \quad
      \abs{\delta(\nu)}=\abs{\delta(\mu)}+1
    \]
\end{enumerate}
Each one of these transformations decreases \(\abs{\mu}+\abs{\delta}\) by \(1\).
\end{proof}

\begin{lemm}
  \label{lemm:fibration}
  (Let \(X\) be a point.)
  Let \(\nu\subseteq\mu\) be an admissible partition. The flags \(V_{\bullet}\in F_{\mu}(E_{\bullet})\) with \(\nu(V_{\bullet})=\nu\) form a smooth fibration over \(\mathring{F}_{\nu}(E_{\bullet})\).
  Moreover, if there are \(p\) elements in the chain of admissible partitions from \(\nu\) to \(\mu\) with respect to the partial order \(\subseteq\), then the relative dimension of this fibration is strictly less than \(p\).
\end{lemm}
\begin{proof}
  By Gaussian elimination with respect to the reference flag \(E_{\bullet}\), the flag \(V_{\bullet}\) is generated by some vectors \(\Set{v_{1},\dotsc,v_{d}}\) with \(v_{d+1-i}\in E_{n_{i}}\setminus E_{n_{i}-1}\) for pairwise distinct integers \(n_{i}\). Since \(V_{d}=\langle v_{1},\dotsc,v_{d}\rangle\), one has \(\Set{n_{d},\dotsc,n_{1}}=\Set{\nu_{d},\dotsc,\nu_{1}}\). However the order of the terms may differ.
  This allows us to define a map 
  \[
    \Set*{V_{\bullet}\in F_{\mu}(E_{\bullet})\colon\nu(V_{\bullet})=\nu}
    \to
    \mathring{F}_{\nu}(E_{\bullet}),
  \]
  given by
  \[
    V_{\bullet}
    \mapsto
    W_{\bullet}=(V_{d}\cap E_{\nu_{d+1-i}})_{i=1,\dotsc,d}.
  \]
  The fiber of this map over a flag \(W_{\bullet}\) is then given by a tower of projectivized bundles
  \[
    \P(W_{k_{d}}/S_{d-1})
    \to
    \dotsb
    \to
    \P(W_{k_{2}}/S_{1})
    \to
    \P(W_{k_{1}})
    \to
    \Set{W_{\bullet}},
  \]
  where we denote by \(S_{i}/S_{i-1}\) the tautological subbundle at step \(i\),
  and
  where for \(j=1,\dotsc,d\), the integer \(k_{j}\) is the largest \(i\) for which \(\nu_{d+1-i}\leq\mu_{d+1-j}\).
  Note that \(k_{j}\geq j\), since \(\nu\subseteq\mu\).
  The relative dimension of the fibration is
  \[
    (k_{1}-1)+\dotsb+(k_{d}-d).
  \]

  If \(\nu=\mu\), this quantity is of course \(0\).
  As in the proof of the Lemma~\ref{lemm:dimension}, the partition \(\nu\) being fixed, let us consider the three ways to get an immediate predecessor of an admissible partition \(\mu\) (we keep the numerotation of \ref{lemm:dimension}).
  \begin{enumerate}
    \item
      The integers \(k_{j}\) will be unchanged unless the black dot on the left hand side is a part of \(\nu\), in which case one of the integers \(k_{j}\) decreases by \(1\).
    \item[2a.]
      The integers \(k_{j}\) will be unchanged unless \(n+1\) is a part of \(\nu\), in which case one of the integers \(k_{j}\) decreases by \(1\).
    \item[2b.]
      The integers \(k_{j}\) will be unchanged unless \((n+1-i)\) or \((n+1+i)\) is a part of \(\nu\), in which case one or two of the integers \(k_{j}\) decreases by \(1\).
  \end{enumerate}
  In particular, if \(\mu\) is a direct successor of \(\nu\), then the quantity
  \((k_{1}-1)+\dotsb+(k_{d}-d)\)
  does not change when passing from \(\mu\) to \(\nu\).
  Therefore, if there are no occurences of step 2b., where it decreases by \(2\), the result is proven.

  Assume now that in the chain from \(\mu\) to \(\nu\), there are occurrences of step 2b., where both \((n+1+i)\) and \((n+1-i)\) are parts of \(\nu\). 
  Notice that by admissibility of the partition \(\nu\), in that case neither \((n-i)\) nor \((n+i)\) is a part of \(\nu\). Consider the next steps in the remaining chain to \(\nu\) involving \((n-i)\) and \((n+i)\).
  Since \((n-i)+(n+i)<2n+2\) the parts are not symmetric with respect to \((n+1)\), and there should be two separate steps.

  Consider one of these steps, for \(m=n\pm i\):
  \begin{itemize}
    \item
      Either it is a step of type 1., 2a, or 2b., where the quantity \((k_1-1)+\dotsb+(k_d-d)\) does not vary;
    \item
      Or it is a step of type 2b., where \((2n+2-m)\in\nu\). 
  \end{itemize}
  In the latter case, by admissibility, the part \((m-1)\) in the output is again not a part of the partition \(\nu\).
  One can look further in the remaining chain to \(\nu\) with \(m=m-1\).
  Since there cannot be an infinite sequence of steps where one obtains parts that are not in \(\nu\),  at some point the first case will eventually occur.

  Now, one can separate \((n+i)\) and \((n-i)\) in the reasoning, because step \(2b.\) can involve only parts \(m'\) and \(m''\) such that \(m'+m''=2n+2\), whereas in our case, all parts at stake satisfy \(m'+m''\leq(n-i)+(n+i)\). In the end one gets \(\ell+3\) steps where the quantity \((k_1-1)+\dotsb+(k_d-d)\) increases by \(\ell+2\). 

    Furthermore, our proof never involved another step 2b., where both parts are parts of \(\nu\). So one can apply the reasoning to all such steps separately.

  This analysis yields that in a chain of length \(p\) from \(\mu\) to \(\nu\), the quantity
  \((k_{1}-1)+\dotsb+(k_{d}-d)\)
  decreases by strictly less than \(p\).
\end{proof}

\section{Gysin formulas}
\label{se:Gysin}
To sum up the construction in the proof of Theorem~\ref{theo:KL},
for a strict partition \(\mu\subseteq\rho\),
we get a sequence of Kempf--Laksov flag bundles
\[
  F_{(\mu_{1},\dotsc,\mu_{d})}(E_{\bullet})
  \longrightarrow
  F_{(\mu_{2},\dotsc,\mu_{d})}(E_{\bullet})
  \longrightarrow
  \dotsb\longrightarrow
  F_{(\mu_{d-1},\mu_{d})}(E_{\bullet})
  \longrightarrow
  F_{(\mu_{d})}(E_{\bullet})
  \longrightarrow
  X,
\]
induced by forgetful maps,
which is the same as the chain of zero-loci in projective bundles:
\margin{
  \begin{equation}
    \label{eq:chain}
    \begin{tikzpicture}[baseline=(current bounding box).center,]
      \node (zd) at (2,0) {$Z_{d}$};
      \node (zd-1) at (5,0) {$Z_{d-1}$};
      \node(z2) at (9,0) {$Z_{2}$};
      \node (z1) at (12,0) {$Z_{1}$};
      \node (x) at (14,0) {$X$,};
      \draw[->] (zd)--(zd-1);
      \draw[->] (z2)--(z1);
      \draw[->] (z1)--(x);
      \node(pd) at (2,-1.2) {$\P(\iota_{d-1}^{\ast}E_{\mu_{1}}/U_{d-1})$};
      \node (pd-1) at (5,-1.2) {$\P(\iota_{d-2}^{\ast}E_{\mu_{2}}/U_{d-2})$};
      \node (p2) at (9,-1.2) {$\P(\iota_{1}^{\ast}E_{\mu_{d-1}}/U_{1})$};
      \node (p1) at (12,-1.2) {$\P(E_{\mu_{d}})$};
      \draw[right hook-latex] (zd)--node[midway,left,scale=.8]{$\iota_{d}$}(pd);
      \draw[right hook-latex] (zd-1)--node[midway,left,scale=.8]{$\iota_{d-1}$}(pd-1);
      \draw[right hook-latex] (z2)--node[midway,left,scale=.8]{$\iota_{2}$}(p2);
      \path (z1)--node[midway,left,scale=.8]{$\iota_{1}$}(p1) node[midway,sloped]{$=$};
      \draw[->] (pd)--node[near start,above left,scale=.8]{$p_{d}$}(zd-1);
      \draw[->] (p2)--(z1);
      \draw[->] (p1)--(x);
      \draw[dashed,shorten >=1cm] (zd-1)--(z2);
      \draw[dashed,->,shorten <=1cm] (zd-1)--(z2);
      \draw[dashed,shorten >=1.15cm](pd-1)--(z2);
      \draw[->,dashed,shorten <=1.5cm](pd-1)--(z2);
    \end{tikzpicture}
  \end{equation}
}
where for \(i=1,\dotsc,d\), the subvariety \(Z_{i}\bydef\Set{\ell\subseteq(U_{i-1})^{\omega}}\) is
the zero-locus of a regular section of the vector bundle \(L\otimes(U_{i}/U_{i-1})^{\vee}\otimes(U_{i-1}/U_{\delta_{d+1-i}})^{\vee}\).
In the spirit of~\cite{DP1,DP2}, we shall deduce a Gysin formula
for \(\vartheta_{\mu}:F_{\mu}(E_{\bullet})\to X\)
from this description.

We fix an integer \(d\) and we denote by \(U\) the universal subbundle on \(\G_{d}^{\omega}(E)\),
as well as its pullback to \(\F^{\omega}(1,\dotsc,d)(E)\) by the natural forgetful map \(\F^{\omega}(1,\dotsc,d)(E)\to\G_{d}^{\omega}(E)\).
We still denote bu \(U\) the restrictions of these respective bundles to Schubert bundles in \(\G_{d}^{\omega}(E)\) or to Kempf--Laksov bundles in \(\F^{\omega}(1,\dotsc,d)(E)\).
For a symmetric polynomial \(f\) in \(d\) variables, we write \(f(U)\) for the specialization of \(f\) with Chern roots of \(U^{\vee}\).

For a Laurent polynomial \(P\) in \(d\) variables \(t_{1},\dotsc,t_{d}\), and a monomial \(m\), we denote by \([m](P)\) the coefficient of \(m\) in the expansion of \(P\). Clearly, for any second monomial \(m'\), one has \([mm'](Pm')=[m](P)\).

\begin{theo}[Gysin formula]
  \label{thm:gysin-KL}
  For a strict partition \(\mu\subseteq\rho\) and \(\vartheta_{\mu}\colon F_{\mu}(E_{\bullet})\to X\):
  \[
    (\vartheta_{\mu})_{\ast} f(U)
    =
    \big[\tprod_{j=1}^{d}t_{j}^{\mu_{j}-1}\big]
    \Big(
      f(t_{1},\dotsc,t_{d})
      \tprod_{1\leq i<j\leq d}(t_{i}-t_{j})
      \tprod_{\substack{1\leq i<j\leq d\\\mu_{i}+\mu_{j}>2n+1}}(c_{1}(L)+t_{i}+t_{j})
      \tprod_{1\leq j\leq d}s_{1/t_{j}}(E_{\mu_{j}})
    \Big).
  \]
\end{theo}
\begin{proof}
  We will prove this formula by induction on \(d\).
  With the notation of~\eqref{eq:chain},
  for \(i=1,\dotsc,d\), let
  \[
    \xi_{i}
    \bydef
    c_{1}\left((U_{d+1-i}/U_{d-i})^{\vee}\right)
    \in A^{\bullet}(Z_{d+1-i}).
  \]
  Then (the pullbacks of) \(\xi_{1},\dotsc,\xi_{d}\) form a set of Chern roots for \(U^{\vee}\) on \(Z_{d}\simeq F_{\mu}(E_{\bullet})\).

  We want to compute \((\vartheta_{\mu})_{\ast}f(\xi_{1},\dotsc,\xi_{d})\).
  If \(d=1\), this is the Gysin formula along projective bundles of lines (see~\cite{DP1}).
  Assume that the formula holds for \(d-1\). Since \(Z_{d-1}\simeq F_{\mu_{2},\dotsc,\mu_{d}}(E_{\bullet})\), we know the Gysin formula \(A^{\bullet}(Z_{d-1})\to A^{\bullet}(X)\), it is thus sufficient to study the Gysin map \(A^{\bullet}(Z_{d})\to A^{\bullet}(Z_{d-1})\).
  Considering \eqref{eq:chain}, we decompose this map as
  \[
    A^{\bullet}(Z_{d})
    \stackrel{(\iota_{d})_{\ast}}\longrightarrow
    A^{\bullet}\P(\iota_{d-1}^{\ast}E_{\mu_{1}}/U_{d-1})
    \stackrel{(p_{d})_{\ast}}\longrightarrow
    A^{\bullet}(Z_{d-1}).
  \]
  The Gysin formula for \((p_{d})_{\ast}\) is the formula for projective bundles of lines. It remains to study the Gysin formula for \((\iota_{d})_{\ast}\).

  Recall that the zero-locus \(Z_{d}\) is cut out by a regular section of the vector bundle
  \(L\otimes (U_{d}/U_{d-1})^{\vee}\otimes(U_{d-1}/U_{\delta_{1}})^{\vee}\).
  By \cite[Prop.~14.1]{Fulton}, the map \((\iota_{d})_{\ast}(\iota_{d})^{\ast}\) is given by the cup-product with the top Chern class
  \[
    c_{\text{top}}(L\otimes (U_{d}/U_{d-1})^{\vee}\otimes(U_{d-1}/U_{\delta_{1}})^{\vee})
    =
    \prod_{1<j<d-\delta_{1}}(c_{1}(L)+\xi_{1}+\xi_{j})
    =
    \prod_{j>1\colon\mu_{1}+\mu_{j}>2n+1}(c_{1}(L)+\xi_{1}+\xi_{j}).
  \]

  Composing the Gysin formulas for \((p_{d})_{\ast}\) (see~\cite{DP1}) and \((\iota_{d})_{\ast}\) (and using the projection formula), we get
  \[
    (p_{d})_{\ast}(\iota_{d})_{\ast}(f(\xi_{1},\xi_{2},\dotsc,\xi_{d}))
    =
    [t_{1}^{\mu_{1}-d}]
    \big(
      f(t_{1},\xi_{2},\dotsc,\xi_{d})
      \prod_{j>1\colon\mu_{1}+\mu_{j}>2n+1}\!\!
      (c_{1}(L)+t_{1}+\xi_{j})\;
      s_{1/t_{1}}(E_{\mu_{1}}/U_{d-1})
    \big).
  \]
  Now
  \[
    s_{1/t_{1}}(E_{\mu_{1}}/U_{d-1})
    =
    s_{1/t_{1}}(E_{\mu_{1}})c_{1/t_{1}}(U_{d-1})
    =
    s_{1/t_{1}}(E_{\mu_{1}})\prod_{1<j}(1-\xi_{j}/t_{1})
    =
    s_{1/t_{1}}(E_{\mu_{1}})\prod_{1<j}\frac{(t_{1}-\xi_{j})}{t_{1}}.
  \]
  Thus (multiplying the extracted monomial and the polynomial by \(t_{1}^{d-1}\)):
  \[
    (p_{d})_{\ast}(\iota_{d})_{\ast}(f(\xi_{1},\xi_{2},\dotsc,\xi_{d}))
    =
    [t_{1}^{\mu_{1}-1}]
    \Big(
      f(t_{1},\xi_{2},\dotsc,\xi_{d})
      \tprod_{1<j\leq d}
      (t_{1}-\xi_{j})
      \tprod_{\substack{1<j\leq d\\\mu_{1}+\mu_{j}>2n+1}}\!
      (c_{1}(L)+t_{1}+\xi_{j})\;
      s_{1/t_{1}}(E_{\mu_{1}})
    \Big).
  \]
  Composing \(A^{\bullet}(Z_{d-1})\to A^{\bullet}(X)\) known by induction with this expression for \(A^{\bullet}(Z_{d})\to A^{\bullet}(Z_{d-1})\), one gets the stated formula for \(A^{\bullet}(Z_{d})\to A^{\bullet}(X)\).
\end{proof}
Note that the formula and its proof also holds for general polynomials \(f(\xi_{1},\dotsc,\xi_{d})\), without symmetry.

As a corollary, one gets the following pushforward formula, that was announced in~\cite{DPNote}.
\begin{theo}
  \label{theo:gysin-Schub}
  For a partition \(\lambda\subseteq(2n-d)^{d}\), and \(\varpi_{\lambda}\colon \Omega_{\lambda}\to X\),
  let \(\mu\) be the complementary partition of \(\lambda\) in \(\rho\),
  then:
  \[
    (\varpi_{\lambda})_{\ast} f(U)
    =
    \big[\tprod_{j=1}^{d}t_{j}^{\mu_{j}-1}\big]
    \Big(
      f(t_{1},\dotsc,t_{d})
      \tprod_{1\leq i<j\leq d}(t_{i}-t_{j})
      \tprod_{\substack{1\leq i<j\leq d\\\mu_{i}+\mu_{j}>2n+1}}(c_{1}(L)+t_{i}+t_{j})
      \tprod_{1\leq j\leq d}s_{1/t_{j}}(E_{\mu_{j}})
    \Big).
  \]
\end{theo}

\section{Orthogonal case}
\label{se:Orthogonal}
The arguments are easily adapted to the orthogonal setting, replacing the projective bundles by quadric bundles of isotropic lines and modifying \textit{mutadis mutandis}.

In this case, we get a sequence of Kempf--Laksov flag bundles
\[
  F_{(\mu_{1},\dotsc,\mu_{d})}(E_{\bullet})
  \longrightarrow
  F_{(\mu_{2},\dotsc,\mu_{d})}(E_{\bullet})
  \longrightarrow
  \dotsb\longrightarrow
  F_{(\mu_{d-1},\mu_{d})}(E_{\bullet})
  \longrightarrow
  F_{(\mu_{d})}(E_{\bullet})
  \longrightarrow
  X,
\]
induced by forgetful maps,
which is the same as the chain of zero-loci in quadric bundles:
\margin{
  \begin{equation}
    \begin{tikzpicture}[baseline=(current bounding box).center,]
      \node (zd) at (2,0) {$Z_{d}$};
      \node (zd-1) at (5,0) {$Z_{d-1}$};
      \node(z2) at (9,0) {$Z_{2}$};
      \node (z1) at (12,0) {$Z_{1}$};
      \node (x) at (14,0) {$X$,};
      \draw[->] (zd)--(zd-1);
      \draw[->] (z2)--(z1);
      \draw[->] (z1)--(x);
      \node(pd) at (2,-1.2) {$\Q(\iota_{d-1}^{\ast}E_{\mu_{1}}/U_{d-1})$};
      \node (pd-1) at (5,-1.2) {$\Q(\iota_{d-2}^{\ast}E_{\mu_{2}}/U_{d-2})$};
      \node (p2) at (9,-1.2) {$\Q(\iota_{1}^{\ast}E_{\mu_{d-1}}/U_{1})$};
      \node (p1) at (12,-1.2) {$\Q(E_{\mu_{d}})$};
      \draw[right hook-latex] (zd)--node[midway,left,scale=.8]{$\iota_{d}$}(pd);
      \draw[right hook-latex] (zd-1)--node[midway,left,scale=.8]{$\iota_{d-1}$}(pd-1);
      \draw[right hook-latex] (z2)--node[midway,left,scale=.8]{$\iota_{2}$}(p2);
      \path (z1)--node[midway,left,scale=.8]{$\iota_{1}$}(p1) node[midway,sloped]{$=$};
      \draw[->] (pd)--node[near start,above left,scale=.8]{$p_{d}$}(zd-1);
      \draw[->] (p2)--(z1);
      \draw[->] (p1)--(x);
      \draw[dashed,shorten >=1cm] (zd-1)--(z2);
      \draw[dashed,->,shorten <=1cm] (zd-1)--(z2);
      \draw[dashed,shorten >=1.15cm](pd-1)--(z2);
      \draw[->,dashed,shorten <=1.5cm](pd-1)--(z2);
    \end{tikzpicture}
  \end{equation}
}
where for \(i=1,\dotsc,d\), the subvariety \(Z_{i}\bydef\Set{\ell\subseteq(U_{i-1})^{\omega}}\) is
the zero-locus of a regular section of the vector bundle \(L\otimes(U_{i}/U_{i-1})^{\vee}\otimes(U_{i-1}/U_{\delta_{d+1-i}})^{\vee}\).

For the sake of completeness we state the general Gysin formula for Kempf--Laksov bundles \(\vartheta_{\mu}\) in the orthogonal setting. Notice the change ``\(i\leq j\)'' in the part of the formula related to isotropy; the contribution \(i=j\) appears when passing from the projective bundle of lines to the quadric bundle of isotropic lines. It reflects the fact that a line is not always isotropic: recall that the quadric bundle of isotropic lines is cut out in the projective bundle of lines by a section of the line bundle \(L\otimes(U_i/U_{i-1})^\vee\otimes(U_i/U_{i-1})^\vee\).
\begin{theo}
  For a partition \(\lambda\subseteq(2n-d)^{d}\), and \(\varpi_{\lambda}\colon\Omega_{\lambda}\to X\),
  let \(\mu\) be the complementary partition of \(\lambda\) in \(\rho\),
  then:
  \[
    (\varpi_{\lambda})_{\ast} f(U)
    =
    \big[\tprod_{j=1}^{d}t_{j}^{\mu_{j}-1}\big]
    \Big(
      f(t_{1},\dotsc,t_{d})
      \tprod_{1\leq i<j\leq d}(t_{i}-t_{j})
      \tprod_{\substack{1\leq i\leq j\leq d\\\mu_{i}+\mu_{j}>2n+1}}(c_{1}(L)+t_{i}+t_{j})
      \tprod_{1\leq j\leq d}s_{1/t_{j}}(E_{\mu_{j}})
    \Big).
  \]
\end{theo}

Note that, if the rank is \(2n\) and \(d=n\), we consider both of the two isomorphic connected components of the flag bundle. Thus, if one is interested in only one of the two components, the result should be divided by \(2\).

\backmatter
\paragraph{Acknowledgments}
The author wants to thank Benoît Cadorel for his selfless and friendly support.
He kindly helped to clarify some key steps of the argumentation in Theorem~\ref{theo:KL}.
He also wants to thank Julien Grivaux for his kind ear and for interesting mathematical discussions.

Preparatary work for this paper was done during a postdoctoral position at Polish Academy of Sciences, Warsaw, under the supervision of Professor Pragacz, to whom this paper is dedicated. Piotr Pragacz is warmly thanked for his support, for introducing the author to the subject and for very useful references. He also suggested the name ``Kempf--Laksov flag bundles'' in~\cite{DP2}.

\bibliographystyle{smfalpha}
\bibliography{gysin}
\vfill
\end{document}